\documentclass[11pt,a4paper]{article}
\usepackage[utf8]{inputenc}
\usepackage{lmodern}
\usepackage[T1]{fontenc}
\usepackage[english]{babel}
\usepackage{ifpdf}
\usepackage[left=1in, right=1in, top=1in, bottom=1in]{geometry}
\usepackage[dvipsnames]{xcolor}
\usepackage[colorlinks=true,linkcolor=teal,citecolor=teal]{hyperref}

\usepackage{graphicx, amsmath, amsthm, amssymb, enumitem, mathrsfs} 
\usepackage{subcaption,caption}
\usepackage{booktabs} 

\usepackage[capitalize]{cleveref}
\usepackage[square,numbers]{natbib}
\usepackage{soul}
\usepackage{multicol}
\usepackage{wrapfig}

\usepackage{tikz}
\usepackage{tikz-3dplot} 
\usepackage{pgfplots, pgfplotstable}
\usetikzlibrary{3d} 
\usetikzlibrary{arrows.meta}

\newcommand{\vol}{{\operatorname{vol}}}

\newcommand{\dmat}{\mathcal{D}}
\newcommand{\rmat}{\mathcal{R}}

\newcommand{\kapdis}{\kappa^{\mathsf{s}}}

\newcommand{\ps}{\mathsf{P}(V)}
\newcommand{\wass}{\mathsf{W}_1}
\newcommand{\kaporc}{\kappa^{\mathsf{or},\alpha}}
\newcommand{\kaplly}{\kappa^{\mathsf{lly}}}
\newcommand{\geodesic}{P^\ast}

\newcommand{\emode}{E}

\newcommand{\dist}[1]{\Delta({#1})}
\newcommand{\degc}[1]{\mathrm{d}({#1})}
\newcommand{\degw}[1]{\mathrm{d}_{w}({#1})}

\newtheorem{theorem}{Theorem}[section]
\newtheorem{corollary}[theorem]{Corollary}
\newtheorem{remark}[theorem]{Remark}
\newtheorem{definition}[theorem]{Definition}
\newtheorem{lemma}[theorem]{Lemma}
\newtheorem{proposition}[theorem]{Proposition}

\usepackage{xcolor, soul}

\hypersetup{
	pdftitle={A Comparative Study of Curvature on Trees},
	pdfauthor={S. J. Robertson}
} 

\title{A Comparative Study of Curvature on Trees}
\author{Sawyer Jack Robertson}
\date{}

\begin{document}
    \captionsetup[figure]{labelfont={bf},labelformat={default},labelsep=period,name={Figure}}
	\maketitle

    \begin{abstract}
        There are several interrelated notions of discrete curvature on graphs. Many approaches utilize the optimal transportation metric on its probability simplex or the distance matrix of the graph. In this survey article, we compute formulas for three different types of curvature on graphs. Along the way, we obtain a comparison result for the curvatures under consideration, a degree-diameter theorem for trees, and a combinatorial identity for certain sums of distances on trees.
    \end{abstract}

    {\footnotesize \textbf{\texttt{Keywords:}} curvature on graphs, trees, optimal transportation, graph distance matrix}

    {\footnotesize \textbf{\texttt{MSC2020:}} 05C05, 05C12, 05C21, 05C10}
    \section{Introduction}

    Discrete notions of curvature on graphs have been the subject of active research for many years. \textit{A priori}, there are deep questions that suggest such a concept would be challenging to construct in general; e.g., whether curvature is a property of each node in a graph, each edge, or both; and, how far from a given node or edge one should study the structure of the ambient graph so as to obtain a suitably local formulation of curvature at each particular place. Despite these challenges, there exist many such formulations which satisfy powerful properties (e.g., eigenvalue estimates, Bonnet-Myers-type theorems). In this article, we focus on two families of discrete curvature which are based, respectively, on the optimal transportation metric and on graph distance matrices.
    
    In the former case, we consider Ollivier's Ricci curvature, defined originally for Markov chains on metric spaces~\cite{ollivier2009ricci} (denoted $\kaporc_{ij}$, see \cref{defn:orc}), and the subsequent modification by Lin, Lu, and Yau in~\cite{lin2011ricci} (denoted $\kaplly_{ij}$, see \cref{defn:lly}). Both of these approaches for curvature on graphs make use of the optimal transportation metric between probability measures related to the simple random walk on a graph and, notably, are defined on each edge in a given graph.
    
    Meanwhile, recent years have seen the emergence of several new approaches to discrete curvature on graphs which utilize the shortest path metric (and other metrics, see~\cite{devriendt2022discrete,devriendt2024graph}) defined on the nodes, including that of Steinerberger in~\cite{steinerberger2023curvature} (denoted $\kapdis_{i}$, see \cref{defn:steiner}). By comparison, this formulation defines curvature as a property of each node in a given graph. Each of these notions of curvature is illustrated on a fixed graph with ten nodes in \cref{fig:curvature-comparison}.

    When the underlying graph is a tree, these notions of curvature admit closed-form formulas which depend only on local properties of the graph. Although many such formulas have appeared in the literature, to our knowledge, they have not yet been collected and compared in a single document. In this article, we remedy this situation by providing formulas for each of these notions of curvature on trees, along with direct comparisons between them. We summarize each of the formulas in \cref{fig:formulas}. In the remainder of this section, we state two comparison theorems along with an application of these formulas in the form of a degree-diameter bound.

    Here, a tree is a finite, simple graph which is connected and acyclic. We say that a graph $T=(V, E)$ is a combinatorial tree if it is a tree with all edge weights equal to one, and we say that a node $i\in V$ (resp. $e\in E$) is a leaf (resp. leaf edge) if it has exactly one neighboring node (resp. is incident to a leaf). 

    \renewcommand{\arraystretch}{1.75}
    \begin{figure}[t!]
        \begin{center}\footnotesize
            \hspace*{-.3cm}\begin{tabular}{|c|c|c|}\hline
                \textbf{Curvature} & \textbf{Formula (unweighted)} & \textbf{Location}\\\hline\hline
                $\kaporc_{ij}$, $\alpha\geq 1/2$ & $\frac{2(1-\alpha)}{\dist{i, j}}\left(1/\degc{i} + 1/\degc{j} - 1\right)$ & {\scriptsize \crefrange{thm:orc-curv}{thm:orc-curv-combinatorial}}\\
                $\kaplly_{ij}$ & $\frac{2}{\dist{i, j}}\left(1/\degc{i} + 1/\degc{j} - 1\right)$ & {\scriptsize \crefrange{thm:lly-curv}{thm:lly-curv-combinatorial}}\\
                $\kapdis_i$ & $\frac{n}{n-1}(2-\degc{i})$ & {\scriptsize \cref{cor:kapd}}\\\hline
            \end{tabular}
        \end{center}
        \caption{This table contains formulas for three different notions of discrete curvature assuming the underlying graph is a tree $T=(V, E)$. Here, we take $i, j\in V$ to be any nodes that are not necessarily adjacent. We use $\degc{i}$ to denote the combinatorial degree of $i$, and we denote by $\dist{i, j}$ the shortest path distance between $i, j\in V$. We omit the formula for Ollivier-Ricci curvature when the laziness parameter $\alpha$ lies in $[0, 1/2)$ because of its length when typeset, see \cref{eq:orc-on-trees}.}\label{fig:formulas}
    \end{figure}

    \begin{theorem}\label{thm:comparison-1}
        Let $T=(V,E)$ be a combinatorial tree with $|V|\ge 2$ and let $\{i,j\}\in E$.
        Fix $\alpha\in[0,1)$. Assume that
            \begin{align*}
                \frac{1}{\degc{i}}+\frac{1}{\degc{j}}\ \le\ \frac{1}{1-\alpha}.
            \end{align*}
        For example, $\alpha\ge\frac{1}{2}$ suffices. Then the Ollivier-Ricci curvature, the Lin-Lu-Yau curvature, and the Steinerberger curvature satisfy
            \begin{align}\label{eq:comparison-1}
                \kaporc_{ij} &= (1-\alpha)\,\kaplly_{ij} = (1-\alpha)\,\frac{n-1}{n}\!\left(\frac{\kapdis_{i}}{\degc{i}}+\frac{\kapdis_{j}}{\degc{j}}\right).
            \end{align}
    \end{theorem}

    We provide a proof of~\cref{thm:comparison-1} in~\cref{sec:comparisons-proof}. Notably, each of these notions of curvature is nonpositive on non-leaves and non-leaf edges; and, with the possible exception of $\kaporc_{ij}$ depending on one's choice of $\alpha$, nonnegative on leaves and their incident edges.

    \begin{theorem}\label{thm:comparison}
        Let $T=(V, E)$ be a combinatorial tree, and assume $|V|\geq 3$. Let $\{i, j\}\in E$ be a fixed edge. 
        \begin{enumerate}[label=(\roman*)]
            \item If $\{i, j\}$ is not a leaf edge and $\alpha\in[0, 1)$, then
                \begin{align*}
                    0\geq \kaporc_{ij} \geq \kaplly_{ij} > \max_{x\in \{i, j\}}\left\{\frac{4}{\degc{x}} \left(\kapdis_{x} -\frac{1}{2} \right)\right\}.
                \end{align*}
            \item If $i$ is a leaf node and $\alpha\in[0, 1)$ is such that $1+\frac{1}{\degc{j}} \leq \frac{1}{1-\alpha}$ (for example, $\alpha\ge\frac{1}{2}$ suffices), then
                \begin{align*}
                    0\leq \kaporc_{ij} \leq \kaplly_{ij} \leq \frac{8}{3}\kapdis_i.
                \end{align*}
        \end{enumerate}
    \end{theorem}

    We provide a proof of~\cref{thm:comparison} in~\cref{sec:comparisons-proof}. The proof of this result is a straightforward consequence of the formulas we obtain and highlight in the subsequent sections. We remark that, to our knowledge, comparative analyses such as~\cref{thm:comparison-1} and~\cref{thm:comparison} have not yet appeared in the literature.

    A second application of our formulas for discrete curvature on trees takes the form of the following degree-diameter theorem. We define the diameter of a tree to be the maximum shortest path distance between any pair of nodes.

    \begin{theorem}\label{thm:degree-diameter}
        Let $T=(V, E)$ be a combinatorial tree with diameter $D$. Then 
            \begin{align*}
                D \geq\frac{n-1}{\sum_{i\in V}|2-\degc{i}|}.
            \end{align*}
    \end{theorem}
    
    This bound is sharp up to constant factors in both extremal cases of path graphs (where it reads $n-1\geq\frac{n-1}{2}$) and star graphs (with one central node and $n-1$ edges to all surrounding nodes, and in which case it reads $2\geq\frac{1}{2}\frac{n-1}{n-2}$). \cref{thm:degree-diameter} follows immediately from \cref{cor:kapd} and the following Reverse Bonnet Myers-type theorem, which we state below.~\cref{thm:degree-diameter} is similar to the diameter bounds of Lesniak~\cite{lesniak1975longest} (see also~\cite{qiao2022relation}) and agrees with theirs up to leading constants.

    \begin{theorem}[Reverse Bonnet-Myers]\label{thm:reverse-bm}
        Let $G=(V, E, w)$ be a graph with diameter $D$. Assume that the shortest path distance matrix $\mathcal{D}$ of $G$ admits a solution $\kapdis\in\mathbb{R}^n$ to the equation $\mathcal{D}\kapdis = n\mathbf{1}$. Then it holds
            \begin{align*}
                \|\kapdis\|_1 \geq \frac{n}{D}.
            \end{align*}
    \end{theorem}

    \begin{figure}[t!]
        \begin{center}
            \begin{subfigure}[b]{2.5in}
                \begin{center}
                    \begin{tikzpicture}[scale=1.25]
	\node (V00) at (0.0, 0.0)[circle,fill,color=black,inner sep=2pt]{};
	\node (V01) at (2.0, 0.0)[circle,fill,color=black,inner sep=2pt]{};
	\node (V02) at (6.0, 1.0)[circle,fill,color=black,inner sep=2pt]{};
	\node (V03) at (5.0, 0.0)[circle,fill,color=black,inner sep=2pt]{};
	\node (V04) at (3.0, 0.0)[circle,fill,color=black,inner sep=2pt]{};
	\node (V05) at (4.0, 0.0)[circle,fill,color=black,inner sep=2pt]{};
	\node (V06) at (6.0, -1.0)[circle,fill,color=black,inner sep=2pt]{};
	\node (V07) at (1.0, 0.0)[circle,fill,color=black,inner sep=2pt]{};
	\node (V08) at (3.0, 1.0)[circle,fill,color=black,inner sep=2pt]{};
	\node (V09) at (4.0, -1.0)[circle,fill,color=black,inner sep=2pt]{};
	\node[rotate=0.0, rectangle, inner sep = 2pt, draw, fill=white] (LAB0007) at (0.5, 0.0) {\tiny 0.5};
	\draw[thin] (V00) -- (LAB0007) -- (V07);
	\node[rotate=0.0, rectangle, inner sep = 2pt, draw, fill=white] (LAB0701) at (1.5, 0.0) {\tiny 0};
	\draw[thin] (V07) -- (LAB0701) -- (V01);
	\node[rotate=0.0, rectangle, inner sep = 2pt, draw, fill=white] (LAB0104) at (2.5, 0.0) {\tiny -0.17};
	\draw[thin] (V01) -- (LAB0104) -- (V04);
	\node[rotate=-90.0, rectangle, inner sep = 2pt, draw, fill=white] (LAB0408) at (3.0, 0.5) {\tiny 0.33};
	\draw[thin] (V04) -- (LAB0408) -- (V08);
	\node[rotate=0.0, rectangle, inner sep = 2pt, draw, fill=white] (LAB0405) at (3.5, 0.0) {\tiny -0.33};
	\draw[thin] (V04) -- (LAB0405) -- (V05);
	\node[rotate=45.0, rectangle, inner sep = 2pt, draw, fill=white] (LAB0203) at (5.5, 0.5) {\tiny 0.33};
	\draw[thin] (V02) -- (LAB0203) -- (V03);
	\node[rotate=-45.0, rectangle, inner sep = 2pt, draw, fill=white] (LAB0306) at (5.5, -0.5) {\tiny 0.33};
	\draw[thin] (V03) -- (LAB0306) -- (V06);
	\node[rotate=0.0, rectangle, inner sep = 2pt, draw, fill=white] (LAB0305) at (4.5, 0.0) {\tiny -0.33};
	\draw[thin] (V03) -- (LAB0305) -- (V05);
	\node[rotate=-90.0, rectangle, inner sep = 2pt, draw, fill=white] (LAB0509) at (4.0, -0.5) {\tiny 0.33};
	\draw[thin] (V05) -- (LAB0509) -- (V09);
\end{tikzpicture}
                \end{center}
            \end{subfigure}\subcaption{Ollivier-Ricci curvature, with $\alpha=1/2$.}
            \begin{subfigure}[b]{2.5in}
                \begin{center}
                    \begin{tikzpicture}[scale=1.25]
	\node (V00) at (0.0, 0.0)[circle,fill,color=black,inner sep=2pt]{};
	\node (V01) at (2.0, 0.0)[circle,fill,color=black,inner sep=2pt]{};
	\node (V02) at (6.0, 1.0)[circle,fill,color=black,inner sep=2pt]{};
	\node (V03) at (5.0, 0.0)[circle,fill,color=black,inner sep=2pt]{};
	\node (V04) at (3.0, 0.0)[circle,fill,color=black,inner sep=2pt]{};
	\node (V05) at (4.0, 0.0)[circle,fill,color=black,inner sep=2pt]{};
	\node (V06) at (6.0, -1.0)[circle,fill,color=black,inner sep=2pt]{};
	\node (V07) at (1.0, 0.0)[circle,fill,color=black,inner sep=2pt]{};
	\node (V08) at (3.0, 1.0)[circle,fill,color=black,inner sep=2pt]{};
	\node (V09) at (4.0, -1.0)[circle,fill,color=black,inner sep=2pt]{};
	\node[rotate=0.0, rectangle, inner sep = 2pt, draw, fill=white] (LAB0007) at (0.5, 0.0) {\tiny 1.0};
	\draw[thin] (V00) -- (LAB0007) -- (V07);
	\node[rotate=0.0, rectangle, inner sep = 2pt, draw, fill=white] (LAB0701) at (1.5, 0.0) {\tiny 0};
	\draw[thin] (V07) -- (LAB0701) -- (V01);
	\node[rotate=0.0, rectangle, inner sep = 2pt, draw, fill=white] (LAB0104) at (2.5, 0.0) {\tiny -0.33};
	\draw[thin] (V01) -- (LAB0104) -- (V04);
	\node[rotate=-90.0, rectangle, inner sep = 2pt, draw, fill=white] (LAB0408) at (3.0, 0.5) {\tiny 0.67};
	\draw[thin] (V04) -- (LAB0408) -- (V08);
	\node[rotate=0.0, rectangle, inner sep = 2pt, draw, fill=white] (LAB0405) at (3.5, 0.0) {\tiny -0.67};
	\draw[thin] (V04) -- (LAB0405) -- (V05);
	\node[rotate=45.0, rectangle, inner sep = 2pt, draw, fill=white] (LAB0203) at (5.5, 0.5) {\tiny 0.67};
	\draw[thin] (V02) -- (LAB0203) -- (V03);
	\node[rotate=-45.0, rectangle, inner sep = 2pt, draw, fill=white] (LAB0306) at (5.5, -0.5) {\tiny 0.67};
	\draw[thin] (V03) -- (LAB0306) -- (V06);
	\node[rotate=0.0, rectangle, inner sep = 2pt, draw, fill=white] (LAB0305) at (4.5, 0.0) {\tiny -0.67};
	\draw[thin] (V03) -- (LAB0305) -- (V05);
	\node[rotate=-90.0, rectangle, inner sep = 2pt, draw, fill=white] (LAB0509) at (4.0, -0.5) {\tiny 0.67};
	\draw[thin] (V05) -- (LAB0509) -- (V09);
\end{tikzpicture}
                \end{center}
            \end{subfigure}\subcaption{Lin-Lu-Yau curvature.}
            \begin{subfigure}[t]{2.5in}
                \begin{center}
                    \begin{tikzpicture}[scale=1.25]
	\node (V00) at (0.0, 0.0){};
	\node (V01) at (2.0, 0.0){};
	\node (V02) at (6.0, 1.0){};
	\node (V03) at (5.0, 0.0){};
	\node (V04) at (3.0, 0.0){};
	\node (V05) at (4.0, 0.0){};
	\node (V06) at (6.0, -1.0){};
	\node (V07) at (1.0, 0.0){};
	\node (V08) at (3.0, 1.0){};
	\node (V09) at (4.0, -1.0){};
	\draw[thin] (V00) -- (V07);
	\draw[thin] (V07) -- (V01);
	\draw[thin] (V01) -- (V04);
	\draw[thin] (V04) -- (V08);
	\draw[thin] (V04) -- (V05);
	\draw[thin] (V02) -- (V03);
	\draw[thin] (V03) -- (V06);
	\draw[thin] (V03) -- (V05);
	\draw[thin] (V05) -- (V09);
	\node[rectangle, draw, inner sep=2pt, fill=white] at (0.0, 0.0) {\tiny 1.11};
	\node[rectangle, draw, inner sep=2pt, fill=white] at (2.0, 0.0) {\tiny 0};
	\node[rectangle, draw, inner sep=2pt, fill=white] at (6.0, 1.0) {\tiny 1.11};
	\node[rectangle, draw, inner sep=2pt, fill=white] at (5.0, 0.0) {\tiny -1.11};
	\node[rectangle, draw, inner sep=2pt, fill=white] at (3.0, 0.0) {\tiny -1.11};
	\node[rectangle, draw, inner sep=2pt, fill=white] at (4.0, 0.0) {\tiny -1.11};
	\node[rectangle, draw, inner sep=2pt, fill=white] at (6.0, -1.0) {\tiny 1.11};
	\node[rectangle, draw, inner sep=2pt, fill=white] at (1.0, 0.0) {\tiny 0};
	\node[rectangle, draw, inner sep=2pt, fill=white] at (3.0, 1.0) {\tiny 1.11};
	\node[rectangle, draw, inner sep=2pt, fill=white] at (4.0, -1.0) {\tiny 1.11};
\end{tikzpicture}
                \end{center}
            \end{subfigure}\subcaption{Steinerberger curvature.}
        \end{center}
        \caption{An illustration of the three different notions of curvature under consideration. The underlying graph is an unweighted tree on ten nodes.}\label{fig:curvature-comparison}
    \end{figure}
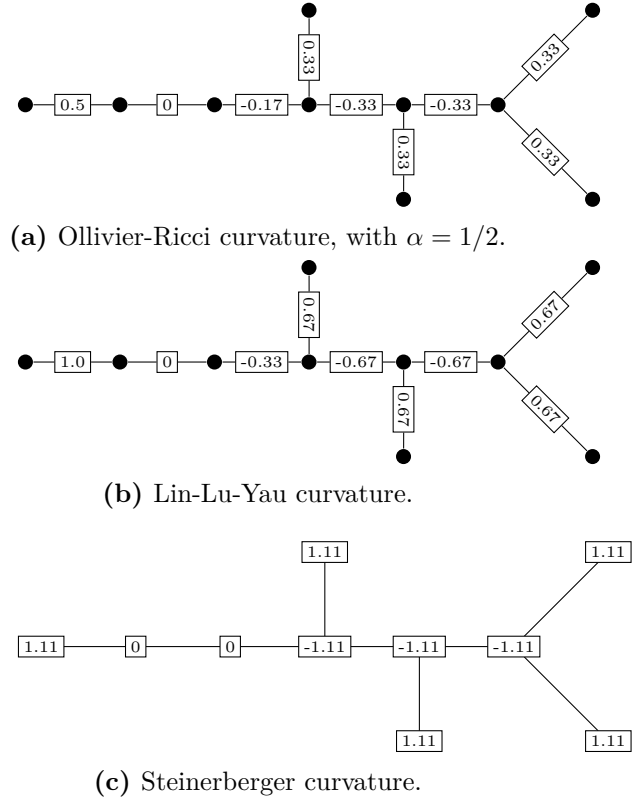

    ~\cref{thm:reverse-bm} is proved in~\cref{sec:comparisons-proof}. We note that this theorem is similar to the reverse Bonnet-Myers theorem in~\cite[Theorem 2]{steinerberger2023curvature}, with two distinctions: first, it is stronger in easing the requirement that the curvature $\kapdis$ of the graph be nonnegative; second, it is weaker in that we assume the existence of a solution to the linear system $\mathcal{D}\kapdis = n\mathbf{1}$, a property which is always satisfied for trees but need not occur in general (this phenomenon is subject to ongoing research by the community, see, e.g.,~\cite{chen2023steinerberger,cushing2024note}).

    Along the way in obtaining the aforementioned results, we prove a combinatorial identity involving sums of distances on graphs (see \cref{lem:dist-identity}).

    \subsection{Related work}

    Trees have regularly made appearances in papers concerning curvature on graphs. Some examples of results concerning Ollivier-Ricci and Lin-Lu-Yau curvatures on trees include the work of Jost and Liu in~\cite{jost2014ollivier} as well as Rubleva~\cite{rubleva2016ricci}. In this paper we opt for a weight-inclusive approach to the lazy random walk measure $m_i^{(\alpha)}$ and thus to our best impression the formulas in~\cref{thm:orc-curv} and~\cref{thm:lly-curv} have not appeared in these exact forms. As for Steinerberger curvature, the case of trees was considered implicitly in~\cite{chen2023steinerberger}, and directly in~\cite{cushing2024note}. In some sense these results follow naturally from historic results which date back to Graham and Lov{\'a}sz~\cite{graham1978distance}. We extend the formulas for curvature on trees to the weighted case. We also include, for the interested reader, a GitHub repository containing a script for calculating these notions of curvature on trees.\footnote{\href{https://github.com/sawyer-jack-1/curvature-on-trees}{https://github.com/sawyer-jack-1/curvature-on-trees}}

    \subsection{Outline of this paper}

    In \cref{sec:notation}, we cover some notation and mathematical background. In \cref{sec:transportation}, we develop formulas for transportation-based notions of curvature on trees. In \cref{sec:distance}, we develop formulas for Steinerberger curvature on trees. Lastly, in \cref{sec:comparisons-proof}, we include the proof of \cref{thm:comparison}.

    \section{Notation and mathematical background}\label{sec:notation}

    We consider graphs of the form $G=(V, E, w)$ where $V = \{1, 2,\dotsc, n\}$ is a set of $n$ nodes, $E\subseteq{V\choose 2}$ is a set of $m$ undirected edges (without loops or multiple edges), and $w = (w_e)_{e\in E}$ is a collection of strictly positive edge weights. We write $w_{ij} = w_{ji} = w_{e}$ interchangeably to refer to the weight of edge $e=\{i, j\}\in E$. We write $i\sim j$ if $\{i, j\}\in E$. We denote by $\degc{i}$ the \emph{degree} of $i\in V$, i.e., the number of edges incident to $i$. We denote by $\degw{i}$ the weighted degree of node $i\in V$:
        \begin{align*}
            \degw{i} &= \sum_{\substack{j\in V \\ j\sim i}} w_{ij}.
        \end{align*}
    Define the volume of the graph by
        \begin{align*}
            \vol(G) &= \sum_{i\in V}\degw{i}.
        \end{align*}
    Note that $\vol(G) = 2\sum_{e\in E}w_e$. We define the \emph{index-oriented edges} $E'$ by the set:
        \begin{align*}
            E' &= \{(i, j) : i,j\in V,\hspace{.1cm}i\sim j,\hspace{.1cm}i<j\},
        \end{align*}
    and the \emph{node-edge oriented incidence matrix} $B\in\mathbb{R}^{n\times m}$ entrywise by the values
        \begin{align}\label{eq:defn-incidence}
            B_{ie_j} &= \begin{cases}
                1 &\text{ if }e_j = (i, \cdot)\\
                -1&\text{ if }e_j = (\cdot, i)\\
                0&\text{ otherwise}
            \end{cases},\hspace{.25cm}i\in V, \hspace{.1cm}e_j\in E'.
        \end{align}
    where $E' = \{e_1,\dotsc, e_m\}$ is any enumeration of the oriented edges. We choose to orient the edges with respect to the indexing on the nodes only for concreteness. It is worth pointing out that $B$ acts on functions $J:E'\rightarrow \mathbb{R}$ by matrix multiplication as follows:
        \begin{align*}
            (BJ)_i &= \sum_{\substack{e\in E'\\ e=(i, \cdot)}} J_e - \sum_{\substack{e\in E'\\ e=(\cdot, i)}} J_e .
        \end{align*}

    A \emph{walk} $P$ is an ordered list of nodes $P = (P_1,\dotsc, P_k)$ such that $P_\ell\in V$ for all $\ell$ and $P_\ell\sim P_{\ell+1}$ for $1\leq \ell \leq k-1$. We say $P$ has length $k-1$. We say $P$ is a \emph{path} (resp. \emph{simple path}) if no node (resp. no node or edge) appears more than once in $P$. $G$ is \emph{connected} if it contains a path between each pair of nodes $i, j\in V$. For $i, j\in V$ we denote by $\dist{i, j}$ the \emph{shortest path distance} between $i, j$, i.e., the (unweighted) length of a shortest path between $i, j$. We denote by $\dmat$ the $n\times n$ matrix with entries given by $\dmat_{ij} = \dist{i, j}$. If $G$ is a tree and $i, j\in V$ the path between $i, j$ of shortest length is unique and simple. We call this path the \emph{geodesic} from $i$ to $j$ and denote it $\geodesic_{ij}$. 

    For a matrix $B\in\mathbb{R}^{n\times k}$, we denote by $B^+\in\mathbb{R}^{k\times n}$ its \emph{Moore-Penrose inverse}~\cite{ben1963contributions}, which is the unique matrix satisfying the following four properties:
        \begin{align*}
            \textit{(i) }BB^+B = B,\hspace{.1cm}\textit{(ii) }B^+BB^+ = B^+,\hspace{.1cm}\textit{(iii) }BB^+\text{ is symmetric, and }\textit{(iv) }B^+B\text{ is symmetric.}
        \end{align*}
    We denote by $\ps$ the simplex of \emph{probability measures} on $V$, that is,
        \begin{align*}
            \ps &= \left\{\mu\in\mathbb{R}^V : \mu_i \geq 0\text{ for each }i\in V,\hspace{.1cm} \mathbf{1}^T\mu = 1\right\}.
        \end{align*} 
    where $\mathbf{1}\in\mathbb{R}^n$ is the vector of all ones. For $\mu,\nu\in\ps$ we define their \emph{optimal transportation distance}, or 1-Wasserstein distance, by the following linear program:
        \begin{align}\label{eq:coupling-form}
            \wass(\mu,\nu)&= \inf\left\{\sum_{i, j\in V} \pi_{ij} \dist{i, j} : \pi\in\mathbb{R}^{n\times n}, \pi_{ij}\geq 0\text{ for each }i, j\hspace{.1cm}\mathbf{1}^T\pi = \mu^T, \pi\mathbf{1} = \nu\right\}.
        \end{align}
    On graphs $\wass(\mu,\nu)$ admits a formulation as a minimum cost flow problem, namely
        \begin{align}\label{eq:flow-form}
            \wass(\mu, \nu) &= \inf \left\{ \sum_{e\in E'} |J_e| : J\in\mathbb{R}^{m},\hspace{.1cm}BJ = \mu - \nu \right\}
        \end{align}
    The equivalence of \cref{eq:coupling-form} and \cref{eq:flow-form} is nontrivial but can be shown with some care, see, e.g.,~\cite[Chapter 6]{peyre2019computational} for a derivation.

    \subsection{Transportation-based curvatures}

    The first two notions of curvature which we consider utilize the optimal transportation metric $\wass$ as applied to Dirac masses (i.e. nodes) and lazy random walk measures.
    
    Before proceeding, we introduce a bit of notation. For $i\in V$ and $\alpha\in[0, 1)$ define the $\alpha$-lazy random walk measure, denoted $m_i^{(\alpha)}\in\ps$ entrywise by the formula
        \begin{align}\label{eq:lazy-rw-measure}
            m_i^{(\alpha)}(x)&=\begin{cases}
                \alpha &\text{ if }x=i,\\
                \frac{(1-\alpha)w_{ix}}{\degw{i}}&\text{ if }x\sim i,\\
                0&\text{ otherwise,}
            \end{cases}
        \end{align}
    for each $x\in V$. The probability measure $m_i^{(\alpha)}$ is the one-step transition measure of the $\alpha$-lazy simple random walk on $G$ subject to the initial condition of starting at node $i$. We remark that with this setup, it is natural to view the weights $w_{ix}$ as affinities or conductances between nodes, as $m_i^{(\alpha)}(x)$ places more mass on adjacent nodes with greater edge weights. We can now define the Ollivier-Ricci and Lin-Lu-Yau curvature on graphs as follows.

    \begin{figure}[t!]\caption{A visualization of the Ollivier-Ricci curvature $\kaporc_{ij}$ at each edge of a tree on 25 nodes, with various choices of $\alpha$. Edge width is proportional to $1-\kaporc_{ij}$.}\label{fig:visualization-orc}
        \begin{center}
            \hspace*{-0.1cm}
            \input{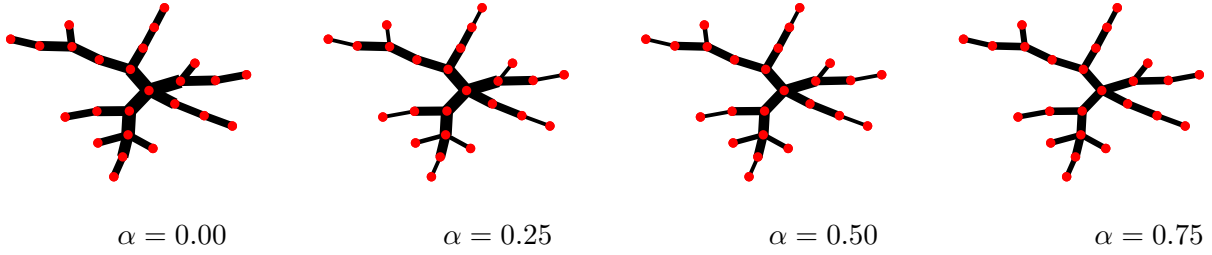}
        \end{center}
    \end{figure}

    \begin{definition}[Ollivier-Ricci curvature~\cite{ollivier2009ricci}]\label{defn:orc}
        Let $G=(V, E, w)$ be any connected graph and $i, j\in V$. We define the {\normalfont Ollivier-Ricci curvature} of $G$ between $i, j$ by the formula
            \begin{align*}
                \kaporc_{ij}&= 1-\frac{\wass(m_i^{(\alpha)}, m_j^{(\alpha)})}{\dist{i, j}}.
            \end{align*}
    \end{definition}

    It is important to note that there is a decent amount of variability from article to article where it concerns the choice of measure $m_i^{(\alpha)}$ with respect to which $\kaporc_{ij}$ is defined.\footnote{In some sense the reason for this is that, in~\cite{ollivier2007ricci}, Ollivier points out that as long as $m_i^{(\alpha)}$ ``depends measurably on $i$,'' it can be used to construct a suitable notion of curvature for a Markov chain on a metric space.} The authors in, e.g.,~\cite{tian2023curvature} opt for a heat kernel-style approach; the authors in~\cite{jost2014ollivier} take the convention $\alpha=0$; and the author of~\cite{rubleva2016ricci} opts to omit the dependence on weights in the definition of $m_i^{(\alpha)}$ and incorporate them solely in the resulting transportation cost. In this paper, our definition follows the conventions of~\cite{Bauer2012ollivier,Bhattacharya2015exact}. We illustrate an example of Ollivier-Ricci curvature in \cref{fig:visualization-orc}.

    In~\cite{lin2011ricci}, the authors Lin, Lu, and Yau proposed an approach to eliminate the dependence on $\alpha$ and proved that the function $\alpha\mapsto\kaporc_{ij}/(1-\alpha)$ is increasing and bounded on $[0, 1)$, and thus admits a limit from the left at $\alpha=1$. We have, based on their approach, the following definition, which is visualized in \cref{fig:visualization-lly}.

    \begin{definition}[Lin-Lu-Yau curvature~\cite{lin2011ricci}]\label{defn:lly}
        Let $G=(V, E, w)$ be any connected graph and $i, j\in V$. We define the {\normalfont Lin-Lu-Yau curvature} of $G$ between $i, j$ by the formula
            \begin{align*}
                \kaplly_{ij}&= \lim_{\alpha\rightarrow 1^-}\frac{\kaporc_{ij}}{1-\alpha}.
            \end{align*}
    \end{definition}

    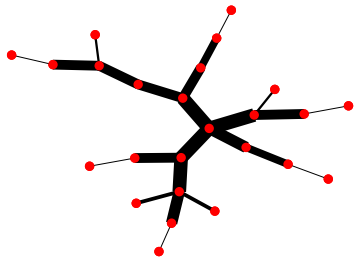
\begin{figure}[h!]\caption{A visualization of the Lin-Lu-Yau curvature $\kaplly_{ij}$ at each edge of a tree on 25 nodes. Edge width is proportional to $1-\kaplly_{ij}$.}\label{fig:visualization-lly}
        \begin{center}
            \hspace*{-0.1cm}
            \begin{tikzpicture}[scale=.5]
	\node (V00) at (4.636, 0.931)[circle,fill,color=red,inner sep=1.25pt]{};
	\node (V01) at (7.716, 2.489)[circle,fill,color=red,inner sep=1.25pt]{};
	\node (V02) at (2.721, 5.094)[circle,fill,color=red,inner sep=1.25pt]{};
	\node (V03) at (2.615, 5.914)[circle,fill,color=red,inner sep=1.25pt]{};
	\node (V04) at (5.828, 5.823)[circle,fill,color=red,inner sep=1.25pt]{};
	\node (V05) at (4.302, 0.18)[circle,fill,color=red,inner sep=1.25pt]{};
	\node (V06) at (5.78, 1.249)[circle,fill,color=red,inner sep=1.25pt]{};
	\node (V07) at (0.405, 5.377)[circle,fill,color=red,inner sep=1.25pt]{};
	\node (V08) at (8.78, 2.097)[circle,fill,color=red,inner sep=1.25pt]{};
	\node (V09) at (4.839, 1.764)[circle,fill,color=red,inner sep=1.25pt]{};
	\node (V10) at (4.931, 4.237)[circle,fill,color=red,inner sep=1.25pt]{};
	\node (V11) at (1.499, 5.124)[circle,fill,color=red,inner sep=1.25pt]{};
	\node (V12) at (7.366, 4.469)[circle,fill,color=red,inner sep=1.25pt]{};
	\node (V13) at (6.606, 2.935)[circle,fill,color=red,inner sep=1.25pt]{};
	\node (V14) at (3.752, 4.6)[circle,fill,color=red,inner sep=1.25pt]{};
	\node (V15) at (4.889, 2.663)[circle,fill,color=red,inner sep=1.25pt]{};
	\node (V16) at (8.142, 3.819)[circle,fill,color=red,inner sep=1.25pt]{};
	\node (V17) at (5.406, 5.038)[circle,fill,color=red,inner sep=1.25pt]{};
	\node (V18) at (2.467, 2.437)[circle,fill,color=red,inner sep=1.25pt]{};
	\node (V19) at (6.216, 6.564)[circle,fill,color=red,inner sep=1.25pt]{};
	\node (V20) at (6.821, 3.791)[circle,fill,color=red,inner sep=1.25pt]{};
	\node (V21) at (5.63, 3.437)[circle,fill,color=red,inner sep=1.25pt]{};
	\node (V22) at (3.663, 2.653)[circle,fill,color=red,inner sep=1.25pt]{};
	\node (V23) at (9.315, 4.034)[circle,fill,color=red,inner sep=1.25pt]{};
	\node (V24) at (3.701, 1.457)[circle,fill,color=red,inner sep=1.25pt]{};
	\draw[thin, line width=0.000mm] (V00.center) -- (V05.center);
	\draw[thin, line width=1.500mm] (V00.center) -- (V09.center);
	\draw[thin, line width=0.000mm] (V01.center) -- (V08.center);
	\draw[thin, line width=1.000mm] (V01.center) -- (V13.center);
	\draw[thin, line width=0.300mm] (V02.center) -- (V03.center);
	\draw[thin, line width=1.300mm] (V02.center) -- (V11.center);
	\draw[thin, line width=1.300mm] (V02.center) -- (V14.center);
	\draw[thin, line width=0.000mm] (V04.center) -- (V19.center);
	\draw[thin, line width=1.000mm] (V04.center) -- (V17.center);
	\draw[thin, line width=0.500mm] (V06.center) -- (V09.center);
	\draw[thin, line width=0.000mm] (V07.center) -- (V11.center);
	\draw[thin, line width=1.800mm] (V09.center) -- (V15.center);
	\draw[thin, line width=0.500mm] (V09.center) -- (V24.center);
	\draw[thin, line width=1.300mm] (V10.center) -- (V14.center);
	\draw[thin, line width=1.300mm] (V10.center) -- (V17.center);
	\draw[thin, line width=1.800mm] (V10.center) -- (V21.center);
	\draw[thin, line width=0.300mm] (V12.center) -- (V20.center);
	\draw[thin, line width=1.500mm] (V13.center) -- (V21.center);
	\draw[thin, line width=1.300mm] (V15.center) -- (V22.center);
	\draw[thin, line width=1.800mm] (V15.center) -- (V21.center);
	\draw[thin, line width=0.000mm] (V16.center) -- (V23.center);
	\draw[thin, line width=1.300mm] (V16.center) -- (V20.center);
	\draw[thin, line width=0.000mm] (V18.center) -- (V22.center);
	\draw[thin, line width=1.800mm] (V20.center) -- (V21.center);
	\node at (4.636, 0.931)[circle,fill,color=red,inner sep=1.25pt]{};
	\node at (7.716, 2.489)[circle,fill,color=red,inner sep=1.25pt]{};
	\node at (2.721, 5.094)[circle,fill,color=red,inner sep=1.25pt]{};
	\node at (2.615, 5.914)[circle,fill,color=red,inner sep=1.25pt]{};
	\node at (5.828, 5.823)[circle,fill,color=red,inner sep=1.25pt]{};
	\node at (4.302, 0.18)[circle,fill,color=red,inner sep=1.25pt]{};
	\node at (5.78, 1.249)[circle,fill,color=red,inner sep=1.25pt]{};
	\node at (0.405, 5.377)[circle,fill,color=red,inner sep=1.25pt]{};
	\node at (8.78, 2.097)[circle,fill,color=red,inner sep=1.25pt]{};
	\node at (4.839, 1.764)[circle,fill,color=red,inner sep=1.25pt]{};
	\node at (4.931, 4.237)[circle,fill,color=red,inner sep=1.25pt]{};
	\node at (1.499, 5.124)[circle,fill,color=red,inner sep=1.25pt]{};
	\node at (7.366, 4.469)[circle,fill,color=red,inner sep=1.25pt]{};
	\node at (6.606, 2.935)[circle,fill,color=red,inner sep=1.25pt]{};
	\node at (3.752, 4.6)[circle,fill,color=red,inner sep=1.25pt]{};
	\node at (4.889, 2.663)[circle,fill,color=red,inner sep=1.25pt]{};
	\node at (8.142, 3.819)[circle,fill,color=red,inner sep=1.25pt]{};
	\node at (5.406, 5.038)[circle,fill,color=red,inner sep=1.25pt]{};
	\node at (2.467, 2.437)[circle,fill,color=red,inner sep=1.25pt]{};
	\node at (6.216, 6.564)[circle,fill,color=red,inner sep=1.25pt]{};
	\node at (6.821, 3.791)[circle,fill,color=red,inner sep=1.25pt]{};
	\node at (5.63, 3.437)[circle,fill,color=red,inner sep=1.25pt]{};
	\node at (3.663, 2.653)[circle,fill,color=red,inner sep=1.25pt]{};
	\node at (9.315, 4.034)[circle,fill,color=red,inner sep=1.25pt]{};
	\node at (3.701, 1.457)[circle,fill,color=red,inner sep=1.25pt]{};
\end{tikzpicture}
        \end{center}
    \end{figure}

    \subsection{Distance-based curvatures}

    In this subsection we introduce the curvature formulation of Steinerberger, defined as follows.

    \begin{definition}[Steinerberger curvature~\cite{steinerberger2023curvature}]\label{defn:steiner}
        Let $G=(V, E, w)$ be any graph and $i\in V$. Let $\kapdis\in\mathbb{R}^n$ be defined by the equation
            \begin{align}\label{eq:steinerberger-defn}
                \kapdis = \dmat^{+}(n\mathbf{1}).
            \end{align}
        For $i\in V$, we define the {\normalfont Steinerberger curvature of $G$ at $i$}, denoted $\kapdis_{i}$, to be given by the $i$-th entry of $\kapdis$.
    \end{definition}

    Ordinarily one must exercise a bit of care in defining $\kapdis$ as the solution to the equation $\mathcal{D}\kapdis = n\mathbf{1}$ since this may not in general admit any solution, and when it does, the solution will in general be nonunique. Therefore in this paper we adopt the usage of $\mathcal{D}^+$ for concreteness. However for trees, the distance matrix $\mathcal{D}$ is nonsingular~\cite{graham1978distance} and thus in fact $\mathcal{D}^+ = \mathcal{D}^{-1}$, so this does not present any difficulties. Note also the lack of dependence of~\cref{eq:steinerberger-defn} on the weights of the underlying graph per our conventional choice of shortest path distance without regard to edge weights. In the tree setting, the weights contribute only a global scaling factor and thus our convention avoids the loss of any generality (see~\cref{rmk:steinerberger-weights}).

    \begin{remark}\normalfont
        We remark at this juncture that two other variants of curvature on graphs, namely that of Devriendt and Lambiotte~\cite{devriendt2022discrete}, as well as that of Devriendt, Ottolini, and Steinerberger~\cite{devriendt2024graph}, are highly related to the Steinerberger curvature which comes into focus in this article. These notions, up to a scaling factor, are constructed in the same manner as $\kapdis_i$ except with the usage of the effective resistance matrix $\rmat$ in lieu of the distance matrix $\dmat$. As it happens, when the underlying graph is a tree, all three of these formulations are identical up to various (positive) scaling factors which depend possibly on the graph and its weights. Therefore, the formulas and comparative results obtained for $\kapdis_i$ may be readily extended to these cases as well.
    \end{remark}

    
\section{Formulas for transportation-based curvatures}\label{sec:transportation}
    In this section we obtain formulas for the Ollivier-Ricci and Lin-Lu-Yau curvature of weighted trees. The main tool which simplifies our treatment greatly is the following lemma, which is well-known (see, e.g., ~\cite{bapat1997moore, robertson2024all} for additional background). We sketch its proof.

    \begin{proposition}\label{prop:tree-beckmann}
		Let $T=(V, E, w)$ be a weighted tree, $\mu,\nu\in\ps$. For an oriented edge $e = (i, j)\in E'$, define $K_\mu\in\mathbb{R}^{E'}$ by
			\[K_\mu(e= (i, j)) = \sum_{k \in V^\ast(i; e)}\mu(k),\]
		where $V^\ast(i; e)\subset V$ is the set of nodes belonging to the subtree with root $i$ obtained from $T$ by removing the edge $e$ (and similarly for $K_\nu$). Then $J=K_\mu-K_\nu$ is the only solution to the linear system $BJ=\mu-\nu$ and it follows that
			\begin{align}
				\mathcal{W}_1(\mu,\nu) &=\sum_{e\in E'} |K_\mu(e)-K_\nu(e)|.
			\end{align} 
	\end{proposition}

    \begin{proof}
        This follows from a bit of linear algebra. The kernel of $B^T$ is exactly the constant vectors (since, in particular, $T$ is connected) and therefore the range of $B$ is exactly the set of mean-zero vectors. By the rank-nullity theorem, $B$ has trivial kernel and thus the equation $BJ = \mu-\nu$ has exactly one solution in the variable $J$. It suffices to check that $K_\mu-K_\nu$ forms such a solution, which is straightforward.
    \end{proof}

    This fact leads to the real heart of the matter, which is included in the following theorem.

    \begin{theorem}\label{thm:wass-on-trees}
        Let $T=(V, E, w)$ be a weighted tree and let $i,j\in V$ be fixed. We have that $\wass(m^{(\alpha)}_i, m^{(\alpha)}_j)$ can be expressed by the following formulas. First, if $\{i, j\}\in E$, we have
            \begin{align*}
                \wass(m^{(\alpha)}_i, m^{(\alpha)}_j) &= \begin{cases}
                    \sum_{\substack{x\sim i \\ x\neq j}}\frac{w_{ix}(1-\alpha)}{\degw{i}} + \sum_{\substack{y\sim j \\ y\neq i}}\frac{w_{jy}(1-\alpha)}{\degw{j}} \\
                    \quad\quad + |1-(1-\alpha)w_{ij}\left(1/\degw{i} + 1/\degw{j}\right)| &\text{ if } \alpha\in[0, 1/2)\\
                    1 + (1-\alpha)\left(\sum_{x\sim i}\frac{w_{ix} \sigma_{ix}}{\degw{i}} + \sum_{y\sim j}\frac{w_{jy} \sigma_{jy}}{\degw{j}} \right) &\text{ if }\alpha\geq 1/2
                \end{cases}.
            \end{align*}
        If $i\nsim j$, we have
            \begin{align*}
                \wass(m^{(\alpha)}_i, m^{(\alpha)}_j) &= \dist{i, j} + (1-\alpha)\left(\sum_{x\sim i}\frac{w_{ix} \sigma_{ix}}{\degw{i}} + \sum_{y\sim j}\frac{w_{jy}  \sigma_{jy}}{\degw{j}} \right),
            \end{align*}
        where for brevity in both cases we write $\sigma_{ix} = -1$ if $\{i, x\}$ is on the geodesic $\geodesic_{ij}$ and $\sigma_{ix} = 1$ otherwise.
    \end{theorem}

    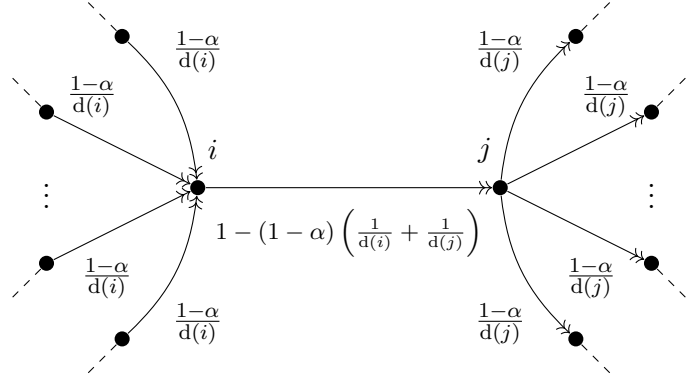
\begin{figure}[t!]  
        \begin{center}
            \begin{tikzpicture}[scale=2]
                \node (1) at  (-2, 0)[circle,fill,color=black,inner sep=2pt]{};
                \node (4) at  (0, 0)[circle,fill,color=black,inner sep=2pt]{};
                
                \node (B1) at  (-2.5, 1)[circle,fill,color=black,inner sep=2pt]{};
                \node (B2) at  (-3, .5)[circle,fill,color=black,inner sep=2pt]{};
                \node (B3) at  (-3, -.5)[circle,fill,color=black,inner sep=2pt]{};
                \node (B4) at  (-2.5, -1)[circle,fill,color=black,inner sep=2pt]{};

                \node (C1) at  (0.5, 1)[circle,fill,color=black,inner sep=2pt]{};
                \node (C2) at  (1, .5)[circle,fill,color=black,inner sep=2pt]{};
                \node (C3) at  (1, -.5)[circle,fill,color=black,inner sep=2pt]{};
                \node (C4) at  (0.5, -1)[circle,fill,color=black,inner sep=2pt]{};

                \draw[color=black, ->>] (1) to (4) node[right]{};

                \node at (-1.9, 0.25) {$i$};
                \node at (-0.1, 0.25) {$j$};

                \draw[color=black, ->>, bend left = 20] (B1) to (1) node[right]{};
                \draw[color=black, ->>] (B2) to (1) node[right]{};
                \draw[color=black, ->>] (B3) to (1) node[right]{};
                \draw[color=black, ->>, bend right = 20] (B4) to (1) node[right]{};
                \node at (-3, 0) {$\vdots$};

                \draw[color=black, <<-, bend right = 20] (C1) to (4) node[right]{};
                \draw[color=black, <<-] (C2) to (4) node[right]{};
                \draw[color=black, <<-] (C3) to (4) node[right]{};
                \draw[color=black, <<-, bend left = 20] (C4) to (4) node[right]{};
                \node at (1, 0) {$\vdots$};

                \node at (-2, .9) {$\frac{1-\alpha}{\degc{i}}$};
                \node at (-2.7, .6) {$\frac{1-\alpha}{\degc{i}}$};
                \node at (-2, -.9) {$\frac{1-\alpha}{\degc{i}}$};
                \node at (-2.6, -.6) {$\frac{1-\alpha}{\degc{i}}$};

                \node at (-1, -0.3) {{\footnotesize $1-(1-\alpha)\left(\frac{1}{\degc{i}}+\frac{1}{\degc{j}}\right)$}};

                \node at (0, .9) {$\frac{1-\alpha}{\degc{j}}$};
                \node at (0.7, .6) {$\frac{1-\alpha}{\degc{j}}$};
                \node at (0, -.9) {$\frac{1-\alpha}{\degc{j}}$};
                \node at (0.6, -.6) {$\frac{1-\alpha}{\degc{j}}$};

                \draw[dashed] (B1) -- +(-0.25, 0.25);
                \draw[dashed] (B2) -- +(-0.25, 0.25);
                \draw[dashed] (B3) -- +(-0.25, -0.25);
                \draw[dashed] (B4) -- +(-0.25, -0.25);

                \draw[dashed] (C1) -- +(0.25, 0.25);
                \draw[dashed] (C2) -- +(0.25, 0.25);
                \draw[dashed] (C3) -- +(0.25, -0.25);
                \draw[dashed] (C4) -- +(0.25, -0.25);

            \end{tikzpicture}
        \end{center}        

        \caption{An illustration of the feasible flow between $m^{(\alpha)}_i$ and $m^{(\alpha)}_j$ in the case where $i, j$ are adjacent. For simplicity we take the tree to be unweighted. The nodes $i, j$ are rendered left and right of center, respectively, and all other labels indicate edge flow values. Orientations are indicated locally with arrows, and all other orientations on edges are immaterial.}\label{fig:tree-transport-illustration-adjacent}
    \end{figure}

    \begin{proof}
        Using \cref{prop:tree-beckmann}, it suffices to calculate the different $K_{m^{(\alpha)}_i} - K_{m^{(\alpha)}_j}$ at each edge and then evaluate its weighted 1-norm. Let the unique simple path from $i$ to $j$ be given by an ordered list of nodes
            \begin{align}\label{eq:geodesic-ij}
                \geodesic_{ij} = (i=i_1,i_2,\dotsc, i_{k}=j),
            \end{align}
        where $\dist{i, j} = k-1$. Let $N^\ast(i)\subseteq E'$ denote the subset of edges incident to $i$ which do not lie on $\geodesic_{ij}$ and similarly for $N^\ast(j)$. Without loss of generality, take the orientation $E'$ on the edges to have the following form. For each $e\in N^\ast(i)$ we take $e$ to be written $e=(\cdot, i)$; for each $e\in N^\ast(j)$ we take $e$ to be written $e=(j, \cdot)$, and for each $e\in \geodesic_{ij}$ we let $e$ be oriented according to the index appearance in \cref{eq:geodesic-ij}. All other edges can have arbitrary orientation and any such choice will not affect the calculations to follow. Having set things up thusly, we can separate the calculations into two cases-- that of adjacency and otherwise.

        \textbf{(Case $i\sim j$)} Assuming $\{i, j\}\in E$, we have that upon inspection $K_{m^{(\alpha)}_i} - K_{m^{(\alpha)}_j}$ satisfies
            \begin{align*}
                K_{m^{(\alpha)}_i}(e) - K_{m^{(\alpha)}_j}(e) &= \begin{cases}
                    \frac{w_e(1-\alpha)}{\degw{i}} &\text{ if } e\in N^\ast(i)\\
                    1-(1-\alpha)w_{ij}\left(1/\degw{i} + 1/\degw{j}\right) &\text{ if } e=(i, j)\\
                    \frac{w_e(1-\alpha)}{\degw{j}} &\text{ if } e\in N^\ast(j)\\
                \end{cases}.
            \end{align*}
        See \cref{fig:tree-transport-illustration-adjacent} for an illustration. Therefore we have
            \begin{align*}
                \wass(m^{(\alpha)}_i, m^{(\alpha)}_j) &= \sum_{x\in N^\ast(i)}\frac{w_{ix}(1-\alpha)}{\degw{i}} + \sum_{y\in N^\ast(j)}\frac{w_{jy}(1-\alpha)}{\degw{j}}
                \\
                &\quad\quad + |1-(1-\alpha)w_{ij}\left(1/\degw{i} + 1/\degw{j}\right)|.
            \end{align*}
        The rightmost term will in general be sensitive to the interplay between the weights on the edges near $i, j$ and one's choice of $\alpha$, but by noting first that $1-(1-\alpha)w_{ij}\left(1/\degw{i} + 1/\degw{j}\right) \geq 0$ provided
            \begin{align*}
               w_{ij}\left(1/\degw{i} + 1/\degw{j}\right) &\leq \frac{1}{1-\alpha},
            \end{align*}
        we have that since $ w_{ij}\left(1/\degw{i} + 1/\degw{j}\right) \leq 2$ always, $\alpha\geq 1/2$ is a sufficient condition for\newline $1-(1-\alpha)w_{ij}\left(1/\degw{i} + 1/\degw{j}\right) \geq 0$. Whenever this is the case, it follows that
            \begin{align*}
                \wass(m^{(\alpha)}_i, m^{(\alpha)}_j) &= \sum_{x\in N^\ast(i)}\frac{w_{ix}(1-\alpha)}{\degw{i}} + \sum_{y\in N^\ast(j)}\frac{w_{jy}(1-\alpha)}{\degw{j}}
                \\
                &\quad\quad + (1-(1-\alpha)w_{ij}\left(1/\degw{i} + 1/\degw{j}\right))\\
                &=1 + (1-\alpha)\left(\sum_{x\sim i}\frac{w_{ix} \sigma_{ix}}{\degw{i}} + \sum_{y\sim j}\frac{w_{jy} \sigma_{jy}}{\degw{j}} \right)
            \end{align*}
            where for brevity we write $\sigma_{ix} = -1$ if $\{i, x\}$ is on the geodesic $\geodesic_{ij}$ and $\sigma_{ix} = 1$ otherwise.

        \textbf{(Case $i\nsim j$)} As before but now assuming $\{i, j\}\notin E$, write down the geodesic $\geodesic_{ij}$ in the following enumerated form
            \begin{align*}
                \geodesic_{ij} &= (i = i_1, i_2,\dotsc, i_{k} = j).
            \end{align*}
        Then we have that upon inspection $K_{m^{(\alpha)}_i} - K_{m^{(\alpha)}_j}$ satisfies
            \begin{align*}
                K_{m^{(\alpha)}_i}(e) - K_{m^{(\alpha)}_j}(e) &= \begin{cases}
                    \frac{w_e(1-\alpha)}{\degw{i}} &\text{ if } e\in N^\ast(i)\\
                    1-(1-\alpha)w_{ij}/\degw{i} &\text{ if } e=(i=i_1, i_2)\\
                    1 &\text{ if } e=(i_{\ell \geq 2}, i_{\ell+1})\\
                    1-(1-\alpha)w_{ij}/\degw{j} &\text{ if } e=(i_{k-1}, i_k=j)\\
                    \frac{w_e(1-\alpha)}{\degw{j}} &\text{ if } e\in N^\ast(j)\\
                \end{cases}.
            \end{align*}
        See \cref{fig:tree-transport-illustration-nonadjacent} for an illustration. Therefore we have that 
            \begin{align*}
                \wass(m^{(\alpha)}_i, m^{(\alpha)}_j) &=  \sum_{x\in N^\ast(i)}\frac{w_{ix}(1-\alpha)}{\degw{i}} + \sum_{y\in N^\ast(j)}\frac{w_{jy}(1-\alpha)}{\degw{j}}
                \\
                &\quad\quad+|1-(1-\alpha)w_{i, i_2}/\degw{i}| + |1-(1-\alpha)w_{i_{k-1}, j}/\degw{j}| \\
                &\quad\quad + \dist{i, j} -2
            \end{align*}
        Now we note that since $w_{i, i_2}/\degw{i} \leq 1$ and $w_{i_{k-1}, j}/\degw{j} \leq 1$ always, we have that \newline $1-(1-\alpha)w_{i, i_2}/\degw{i} \geq 0$ and $1-(1-\alpha)w_{i_{k-1}, j}/\degw{j}\geq 0$. Therefore we can simplify the preceding and obtain
            \begin{align*}
                \wass(m^{(\alpha)}_i, m^{(\alpha)}_j) &= \dist{i, j} + (1-\alpha)\left(\sum_{x\sim i}\frac{w_{ix} \sigma_{ix}}{\degw{i}} + \sum_{y\sim j}\frac{w_{jy}  \sigma_{jy}}{\degw{j}} \right)
            \end{align*}
        where for brevity we write $\sigma_{ix} = -1$ if $\{i, x\}$ is on the geodesic $\geodesic_{ij}$ and $\sigma_{ix} = 1$ otherwise.
    \end{proof}

    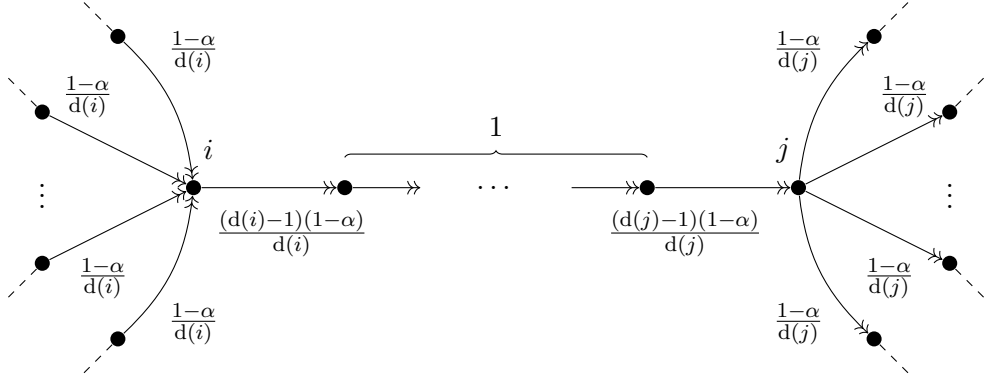
\begin{figure}[t!]  
        \begin{center}
            \begin{tikzpicture}[scale=2]
                \node (1) at  (-2, 0)[circle,fill,color=black,inner sep=2pt]{};
                \node (2) at  (-1, 0)[circle,fill,color=black,inner sep=2pt]{};
                \node (3) at  (1, 0)[circle,fill,color=black,inner sep=2pt]{};
                \node (4) at  (2, 0)[circle,fill,color=black,inner sep=2pt]{};
                
                \node (B1) at  (-2.5, 1)[circle,fill,color=black,inner sep=2pt]{};
                \node (B2) at  (-3, .5)[circle,fill,color=black,inner sep=2pt]{};
                \node (B3) at  (-3, -.5)[circle,fill,color=black,inner sep=2pt]{};
                \node (B4) at  (-2.5, -1)[circle,fill,color=black,inner sep=2pt]{};

                \node (C1) at  (2.5, 1)[circle,fill,color=black,inner sep=2pt]{};
                \node (C2) at  (3, .5)[circle,fill,color=black,inner sep=2pt]{};
                \node (C3) at  (3, -.5)[circle,fill,color=black,inner sep=2pt]{};
                \node (C4) at  (2.5, -1)[circle,fill,color=black,inner sep=2pt]{};

                \draw[color=black, ->>] (1) to (2) node[right]{};
                \draw[color=black, ->>] (2) to (-0.5, 0) node[right]{};
                \draw[color=black, <<-] (3) to (0.5, 0) node[right]{};
                \draw[color=black, ->>] (3) to (4) node[right]{};
                \node at (0, 0) {$\dots$};

                \node at (-1.9, 0.25) {$i$};
                \node at (1.9, 0.25) {$j$};

                \draw[color=black, ->>, bend left = 20] (B1) to (1) node[right]{};
                \draw[color=black, ->>] (B2) to (1) node[right]{};
                \draw[color=black, ->>] (B3) to (1) node[right]{};
                \draw[color=black, ->>, bend right = 20] (B4) to (1) node[right]{};
                \node at (-3, 0) {$\vdots$};

                \draw[color=black, <<-, bend right = 20] (C1) to (4) node[right]{};
                \draw[color=black, <<-] (C2) to (4) node[right]{};
                \draw[color=black, <<-] (C3) to (4) node[right]{};
                \draw[color=black, <<-, bend left = 20] (C4) to (4) node[right]{};
                \node at (3, 0) {$\vdots$};

                \node at (-2, .9) {$\frac{1-\alpha}{\degc{i}}$};
                \node at (-2.7, .6) {$\frac{1-\alpha}{\degc{i}}$};
                \node at (-2, -.9) {$\frac{1-\alpha}{\degc{i}}$};
                \node at (-2.6, -.6) {$\frac{1-\alpha}{\degc{i}}$};

                \node at (-1.35, -.3) {$\frac{(\degc{i}-1)(1-\alpha)}{\degc{i}}$};
                \node at (1.25, -.3) {$\frac{(\degc{j}-1)(1-\alpha)}{\degc{j}}$};

                \node at (2, .9) {$\frac{1-\alpha}{\degc{j}}$};
                \node at (2.7, .6) {$\frac{1-\alpha}{\degc{j}}$};
                \node at (2, -.9) {$\frac{1-\alpha}{\degc{j}}$};
                \node at (2.6, -.6) {$\frac{1-\alpha}{\degc{j}}$};

                \node at (0, 0.4) {$1$};
                \draw[decoration={brace,raise=0},decorate] (-1, 0.2) -- (1, .2);

                \draw[dashed] (B1) -- +(-0.25, 0.25);
                \draw[dashed] (B2) -- +(-0.25, 0.25);
                \draw[dashed] (B3) -- +(-0.25, -0.25);
                \draw[dashed] (B4) -- +(-0.25, -0.25);

                \draw[dashed] (C1) -- +(0.25, 0.25);
                \draw[dashed] (C2) -- +(0.25, 0.25);
                \draw[dashed] (C3) -- +(0.25, -0.25);
                \draw[dashed] (C4) -- +(0.25, -0.25);
            \end{tikzpicture}
        \end{center}        

        \caption{An illustration of the feasible flow between $m^{(\alpha)}_i$ and $m^{(\alpha)}_j$ in the case where $i, j$ are not adjacent. For simplicity we take the tree to be unweighted. The nodes $i, j$ are rendered left and right of center, respectively, and all other labels indicate edge flow values. The dotted line in the center denotes the geodesic between $i, j$ along which (other than the two end-edges), the flow is equal to one. Orientations are indicated locally with arrows, and all other orientations on edges are immaterial.}\label{fig:tree-transport-illustration-nonadjacent}
    \end{figure}

    We now have the following two immediate corollaries which give formulas for the transportation-based curvatures. 

    \begin{corollary}\label{thm:orc-curv}
        Let $T=(V, E, w)$ be a weighted tree and let $i,j\in V$ be fixed. We have that $\kaporc_{ij}$ can be expressed by the following formulas. First, if $\{i, j\}\in E$, we have
            \begin{align}\label{eq:orc-on-trees}
                \kaporc_{ij} &= \begin{cases}
                    1 - \sum_{\substack{x\sim i \\ x\neq j}}\frac{w_{ix}(1-\alpha)}{\degw{i}} - \sum_{\substack{y\sim j \\ y\neq i}}\frac{w_{jy}(1-\alpha)}{\degw{j}} \\
                    \quad\quad - |1-(1-\alpha)w_{ij}\left(1/\degw{i} + 1/\degw{j}\right)| &\text{ if } \alpha\in[0, 1/2)\\
                    - (1-\alpha)\left(\sum_{x\sim i}\frac{w_{ix} \sigma_{ix}}{\degw{i}} + \sum_{y\sim j}\frac{w_{jy} \sigma_{jy}}{\degw{j}} \right) &\text{ if }\alpha\geq 1/2
                \end{cases}.
            \end{align}
        If $i\nsim j$, we have
            \begin{align*}
                \kaporc_{ij} &= -\frac{1-\alpha}{\dist{i, j}}\left(\sum_{x\sim i}\frac{w_{ix} \sigma_{ix}}{\degw{i}} + \sum_{y\sim j}\frac{w_{jy}  \sigma_{jy}}{\degw{j}} \right),
            \end{align*}
        where for brevity in both cases we write $\sigma_{ix} = -1$ if $\{i, x\}$ is on the geodesic $\geodesic_{ij}$ and $\sigma_{ix} = 1$ otherwise.
    \end{corollary}

    Versions of~\cref{thm:orc-curv} have appeared in various settings; e.g.,~\cite{rubleva2016ricci} in the case of weighted shortest path distance and unweighted lazy random walk measures, or~\cite{jost2014ollivier} in the case of locally finite combinatorial trees. In these works and others on general graphs, the main approach to computing the Wasserstein distance has been to utilize the Kantorovich-Rubenstein duality to write~\cref{eq:coupling-form} as an optimization program over Lipschitz functions with bounded Lipschitz constant. Subsequent works such as~\cite{munch2019ollivier} and~\cite{bai2021thesum}, building on~\cite{bourne2018ollivier}, have achieved further simplification by reducing these programs to Lipschitz functions defined only on the neighborhoods of adjacent nodes.

    We state as a corollary the version of the preceding in the case where the tree $T$ is unweighted.

    \begin{corollary}\label{thm:orc-curv-combinatorial}
        Let $T=(V, E)$ be a combinatorial tree and let $i,j\in V$ be fixed. We have that $\kaporc_{ij}$ can be expressed by the following formulas. First, if $\{i, j\}\in E$, we have
            \begin{align}\label{eq:orc-on-trees-comb}
                \kaporc_{ij} &= \begin{cases}
                    2\alpha &\text{ if }\left(1/\degc{i} + 1/\degc{j}\right) > \frac{1}{1-\alpha} \\
                    2(1-\alpha)\left(\frac{1}{\degc{i}} +\frac{1}{\degc{j}} - 1\right) &\text{ if }\left(1/\degc{i} + 1/\degc{j}\right)\leq \frac{1}{1-\alpha}
                \end{cases}.
            \end{align}
        In particular if $\alpha\geq 1/2$, the latter case applies. If $i\nsim j$, we have
            \begin{align*}
                \kaporc_{ij} &= \frac{2(1-\alpha)}{\dist{i, j}}\left(\frac{1}{\degc{i}} +\frac{1}{\degc{j}} - 1\right).
            \end{align*}
    \end{corollary}

    By taking the limit of $\kaporc_{ij}/(1-\alpha)$ as $\alpha\rightarrow 1^-$, the condition $\alpha\geq 1/2$ can be eased and we obtain a straightforward characterization of the Lin-Lu-Yau curvature as well.

    \begin{corollary}\label{thm:lly-curv}
        Let $T=(V, E, w)$ be a weighted tree and let $i,j\in V$ be fixed. We have that $\kaplly_{ij}$ can be expressed by the following formula:
            \begin{align}\label{eq:lly-on-trees}
                \kaplly_{ij} &=  - \frac{1}{\dist{i, j}}\left(\sum_{x\sim i}\frac{w_{ix} \sigma_{ix}}{\degw{i}} + \sum_{y\sim j}\frac{w_{jy} \sigma_{jy}}{\degw{j}} \right)
            \end{align}
        where for brevity we write $\sigma_{ix} = -1$ if $\{i, x\}$ is on the geodesic $\geodesic_{ij}$ and $\sigma_{ix} = 1$ otherwise.
    \end{corollary}

    As before, this admits a combinatorial version which we state below.

    \begin{corollary}\label{thm:lly-curv-combinatorial}
        Let $T=(V, E)$ be a combinatorial tree and let $i,j\in V$ be fixed. We have that $\kaplly_{ij}$ can be expressed by the following formula:
            \begin{align}\label{eq:lly-on-trees-comb}
                \kaplly_{ij} &=  \frac{2}{\dist{i, j}}\left(\frac{1}{\degc{i}} + \frac{1}{\degc{j}} - 1\right).
            \end{align}
    \end{corollary}

    \section{Sums of distances on trees and Steinerberger curvature}\label{sec:distance}

    In this section we prove the following identity for shortest path distances on trees and use it to compute the Steinerberger curvature at each node. Note that this can also be obtained from the formula for the inverse of a tree distance matrix~\cite{graham1978distance}, but we offer a direct combinatorial proof for completeness.

    \begin{lemma}\label{lem:dist-identity}
        Let $T=(V, E)$ be a combinatorial tree and let $i\in V$ be fixed. Then we have the following identity:
            \begin{align*}
                2\sum_{j\in V} \dist{i, j} &= \frac{1}{2}\vol(G) + \sum_{j\in V } \dist{i, j}\degc{j}.
            \end{align*}
    \end{lemma}

    To establish this lemma we use a walk-counting argument which requires some setup and discussion before we write the main proof. Letting $i\in V$ be fixed, we begin by rooting the graph $T$ at node $i$ and organizing its nodes and edges with respect to the root. Specifically, for $j, k\in V$, write $j \prec_i k$ if $j$ appears in $\geodesic_{ik}$ (we identify only $\geodesic_{ii}$ as the ``loop'' $(i, i)$). Note that $(V, \prec_i)$ is a partially ordered set with maximal element $i$. Identify each edge $\{j, k\}\in E(T)$ with the ordered tuple $(j, k)$ where $j\prec_i k$ and write $\emode_i$ to denote the set of all such tuples. If we think of $T$ as being embedded in the plane with the root situated at the top and the remainder of the nodes and edges arranged in a downward-oriented fashion away from the root, the relation $j\prec_i k$ holds if $j$ appears on the branch connecting $i$ to $k$. See \cref{fig:rooted-tree-semil} for an illustration.

    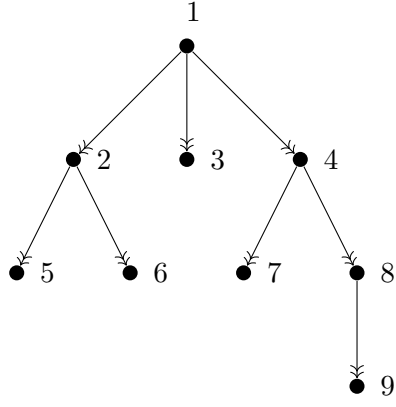
\begin{figure}[t!]
        \begin{center}
            \begin{tikzpicture}[scale=1.5]
                \node (0) at  (0, 0.3){\hspace{5pt}1};
                \node (1) at  (0, 0)[circle,fill,color=black,inner sep=2pt]{};
                \node (2) at (-1, -1)[circle,fill,color=black,inner sep=2pt]{};
                \node (3) at (0, -1)[circle,fill,color=black,inner sep=2pt]{};
                \node (4) at (1, -1)[circle,fill,color=black,inner sep=2pt]{};
                \node (5) at (-1.5, -2)[circle,fill,color=black,inner sep=2pt]{};
                \node (6) at (-0.5, -2)[circle,fill,color=black,inner sep=2pt]{};
                \node (7) at (0.5, -2)[circle,fill,color=black,inner sep=2pt]{};
                \node (8) at (1.5, -2)[circle,fill,color=black,inner sep=2pt]{};
                \node (9) at (1.5, -3)[circle,fill,color=black,inner sep=2pt]{};
    
                \draw[color=black,->>] (1) -- (2) node[right]{\hspace{5pt}2};
                \draw[color=black,->>] (1) -- (3) node[right]{\hspace{5pt}3};
                \draw[color=black,->>] (1) -- (4) node[right]{\hspace{5pt}4};
                \draw[color=black,->>] (2) -- (5) node[right]{\hspace{5pt}5};
                \draw[color=black,->>] (2) -- (6) node[right]{\hspace{5pt}6};
                \draw[color=black,->>] (4) -- (7) node[right]{\hspace{5pt}7};
                \draw[color=black,->>] (4) -- (8) node[right]{\hspace{5pt}8};
                \draw[color=black,->>] (8) -- (9) node[right]{\hspace{5pt}9};
            \end{tikzpicture}
        \end{center}
        \caption{An illustration of the poset $(V, \prec)$ for a tree on $9$ nodes and root $1$. Edge direction indicates orientation in $\emode_1$. For example, $1\prec_1 2$, $4\prec_1 9$, but $2\nprec_1 8$. 
        }\label{fig:rooted-tree-semil}
    \end{figure}

    If $\mathcal{P} = \{P^\alpha\}_{\alpha\in\mathscr{A}}$ is any collection of finite-length walks on the edges of a graph $G$, and $e\in E(G)$, we write
        \begin{align*}
            N_e(\mathcal{P}) &= \sum_{\alpha\in\mathscr{A}} \#\left\{(\ell, \ell+1) : e= (P^\alpha_\ell, P^\alpha_{\ell+1})\right\}
        \end{align*}
    where $\#\{\cdot\}$ indicates set cardinality and $e= (P_\ell, P_{\ell+1})$ is said to hold if $e = \{P_\ell, P_{\ell+1}\}$. We can think of $N_e(\mathcal{P})$ as the number of times that an edge is traversed in the family of walks $\mathcal{P}$. For this paper we will always have that $\# \mathcal{P} <\infty$ so $N_e(\mathcal{P})<\infty$. Define the step count of $\mathcal{P}$ by the expression
        \begin{align}\label{eq:step-count-defn}
            W(\mathcal{P}) &= \sum_{e\in E(G)} N_e(\mathcal{P}).
        \end{align}

    \begin{proof}[Proof of \cref{lem:dist-identity}]
        Assume without loss of generality that $i=1$, and let $\prec$ denote the relation $\prec_1$. For each $j\in V$ with $j\neq 1$, construct a walk $P^j$ by writing $P^j = \geodesic_{ij} + \geodesic_{ji}$ where $+$ indicates concatenation. That is, the walk traverses the geodesic from $i$ to $j$ and then again in reverse. Write $\mathcal{P} = \{P^j\}_{j\neq 1}$. Note upon inspection the immediate identity
            \begin{align}\label{eq:weight-p}
                W(\mathcal{P}) &= \sum_{j\neq 1} 2\dist{i, j},
            \end{align}
        since each geodesic is covered exactly twice and thus contributes its weight as such to the sum $W(\mathcal{P})$. Note also the slightly less obvious recursive identity for $(j, k)\in \emode_1$:
            \begin{align}\label{eq:recursion-p}
                N_{(j, k)}(\mathcal{P}) &= 2 + \sum_{\substack{\ell\sim k \\ k\prec \ell}}N_{(k, \ell)}(\mathcal{P}),
            \end{align}
        This follows from the observation that $(j, k)$ can occur as a step in a walk $P\in\mathcal{P}$ in one of two ways: either $P$ terminates at $k$ before reversing (that is, $P=P^k$), or $P$ proceeds past $k$ and terminates at some $\ell \succ k$ before turning around. If the walk proceeds through $k$, then it must pass through some edge $(k, \ell)$ for $\ell\sim k$, with $(k, \ell) \succ (j, k)$ before reversing as well. We illustrate the collection $\mathcal{P}$ and its edge counts $N_e$ in \cref{fig:walk-1-illustration}. 

        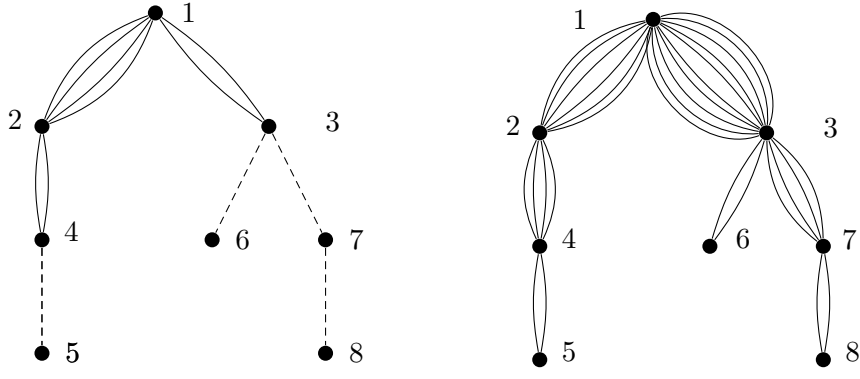
\begin{figure}[t!]\centering

            \begin{subfigure}[t!]{0.4\textwidth}
                \begin{center}
                    \begin{tikzpicture}[scale=1.5]
                        \node (0) at  (0, 0){\hspace{25pt}1};
                        \node (1) at  (0, 0)[circle,fill,color=black,inner sep=2pt]{};
                        \node (2) at  (-1, -1)[circle,fill,color=black,inner sep=2pt]{};
                        \node (3) at  (1, -1)[circle,fill,color=black,inner sep=2pt]{};
                        \node (4) at  (-1, -2)[circle,fill,color=black,inner sep=2pt]{};
                        \node (5) at  (-1, -3)[circle,fill,color=black,inner sep=2pt]{};
                        \node (6) at  (0.5, -2)[circle,fill,color=black,inner sep=2pt]{};
                        \node (7) at  (1.5, -2)[circle,fill,color=black,inner sep=2pt]{};
                        \node (8) at  (1.5, -3)[circle,fill,color=black,inner sep=2pt]{};
        
                        \draw[color=black, bend right = 10] (2) to (4) node[right]{\hspace{5pt}4};
                        \draw[color=black, bend right = 10] (4) to (2) node[right]{};
    
                        \draw[color=black, bend right = 10] (1) to (2) node[right]{\hspace{-25pt}2};
                        \draw[color=black, bend right = 10] (2) to (1) node[right]{};
                        \draw[color=black, bend right = 25] (1) to (2) node[right]{};
                        \draw[color=black, bend right = 25] (2) to (1) node[right]{};
    
                        \draw[color=black, bend right = 10] (1) to (3) node[right]{\hspace{20pt}3};
                        \draw[color=black, bend right = 10] (3) to (1) node[right]{};
        
                        \draw[color=black, densely dashed] (4) to (5) node[right]{\hspace{5pt}5};
                        \draw[color=black, densely dashed] (7) to (8) node[right]{\hspace{5pt}8};
                        \draw[color=black, densely dashed] (3) to (7) node[right]{\hspace{5pt}7};
                        \draw[color=black, densely dashed] (3) to (6) node[right]{\hspace{5pt}6};
                        \draw[color=black, densely dashed] (4) to (5) node[right]{\hspace{5pt}5};
    
                    \end{tikzpicture}
                \end{center}
            \end{subfigure}
            \begin{subfigure}[t!]{0.4\textwidth}
                \begin{center}
                    \begin{tikzpicture}[scale=1.5]
                        \node (0) at  (0, 0){\hspace{-55pt}1};
                        \node (1) at  (0, 0)[circle,fill,color=black,inner sep=2pt]{};
                        \node (2) at  (-1, -1)[circle,fill,color=black,inner sep=2pt]{};
                        \node (3) at  (1, -1)[circle,fill,color=black,inner sep=2pt]{};
                        \node (4) at  (-1, -2)[circle,fill,color=black,inner sep=2pt]{};
                        \node (5) at  (-1, -3)[circle,fill,color=black,inner sep=2pt]{};
                        \node (6) at  (0.5, -2)[circle,fill,color=black,inner sep=2pt]{};
                        \node (7) at  (1.5, -2)[circle,fill,color=black,inner sep=2pt]{};
                        \node (8) at  (1.5, -3)[circle,fill,color=black,inner sep=2pt]{};
        
                        \draw[color=black, bend right = 10] (7) to (8) node[right]{\hspace{5pt}8};
                        \draw[color=black, bend right = 10] (8) to (7) node[right]{};
        
                        \draw[color=black, bend right = 10] (4) to (5) node[right]{\hspace{5pt}5};
                        \draw[color=black, bend right = 10] (5) to (4) node[right]{};
        
                        \draw[color=black, bend right = 10] (2) to (4) node[right]{\hspace{5pt}4};
                        \draw[color=black, bend right = 10] (4) to (2) node[right]{};
                        \draw[color=black, bend right = 25] (2) to (4) node[right]{};
                        \draw[color=black, bend right = 25] (4) to (2) node[right]{};
        
                        \draw[color=black, bend right = 10] (3) to (7) node[right]{\hspace{5pt}7};
                        \draw[color=black, bend right = 10] (7) to (3) node[right]{};
                        \draw[color=black, bend right = 25] (3) to (7) node[right]{};
                        \draw[color=black, bend right = 25] (7) to (3) node[right]{};
        
                        \draw[color=black, bend right = 10] (1) to (2) node[right]{\hspace{-25pt}2};
                        \draw[color=black, bend right = 10] (2) to (1) node[right]{};
                        \draw[color=black, bend right = 25] (1) to (2) node[right]{};
                        \draw[color=black, bend right = 25] (2) to (1) node[right]{};
                        \draw[color=black, bend right = 35] (1) to (2) node[right]{};
                        \draw[color=black, bend right = 35] (2) to (1) node[right]{};
                        
                        \draw[color=black, bend right = 10] (3) to (6) node[right]{\hspace{5pt}6};
                        \draw[color=black, bend right = 10] (6) to (3) node[right]{};
        
                        \draw[color=black, bend right = 10] (1) to (3) node[right]{\hspace{20pt}3};
                        \draw[color=black, bend right = 10] (3) to (1) node[right]{};
                        \draw[color=black, bend right = 20] (1) to (3) node[right]{};
                        \draw[color=black, bend right = 20] (3) to (1) node[right]{};
                        \draw[color=black, bend right = 35] (1) to (3) node[right]{};
                        \draw[color=black, bend right = 35] (3) to (1) node[right]{};
                        \draw[color=black, bend right = 50] (1) to (3) node[right]{};
                        \draw[color=black, bend right = 50] (3) to (1) node[right]{};
                        \draw[color=black, bend right = 65] (1) to (3) node[right]{};
                        \draw[color=black, bend right = 65] (3) to (1) node[right]{};
        
                    \end{tikzpicture}
                \end{center}
            \end{subfigure}
    
            \caption{(\textbf{left}) An illustration of the collection $\mathcal{P}$ from the proof of \cref{lem:dist-identity} on a tree with eight nodes and root $1$. Only the walks $P^2, P^3, P^4$ are shown. Repeated lines indicate the number of times a particular edge appears in the collection $\mathcal{P}$, i.e., $N_e(\mathcal{P})$. Dashed lines indicate unused edges which are taken to be in the surrounding tree. (\textbf{right}) All walks $P^j$ with $j\neq 1$ are shown. }\label{fig:walk-1-illustration}
        \end{figure}
    
        Now we construct a separate collection of walks $\mathcal{Q}$ as follows. For $j\in V$ with $j\neq 1$, and $s=1,2,\dotsc, \degc{j}$, let $Q^{j, s}$ be an identical copy of the geodesic $\geodesic_{j1}$. Write $\mathcal{Q} = \{Q^{j, s}\}_{j, s}$. Note upon inspection the immediate identity
            \begin{align}\label{eq:weight-q}
                W(\mathcal{Q}) &= \sum_{j\neq 1} \dist{i, j} \degc{j}
            \end{align}
        since each geodesic $\geodesic_{j1}$ is covered exactly $\degc{j}$ times and thus contributes its weight as such to the sum $W(\mathcal{Q})$. Note also the slightly less obvious recursive identity for $(j, k)\in \emode_1$:
            \begin{align}\label{eq:recursion-q}
                N_{(j, k)}(\mathcal{Q}) &= \degc{k} + \sum_{\substack{\ell\sim k \\ k\prec \ell}}N_{(k, \ell)}(\mathcal{Q}),
            \end{align}
        This equation holds for the following reason. An edge $(j, k)$ is traversed either by one of the $\degc{k}$ walks instantiated at $k$ going through $j$, or by a walker instantiated at some $\ell \succ k$ going through $(j, k)$. If $\ell$ is such that $(k, \ell)\in\emode_i$, then the number of individuals going through $(k, \ell)$ and then $(j, k)$ is $N_{(k, \ell)}(\mathcal{Q})$. We illustrate the collection $\mathcal{Q}$ and its edge counts $N_e$ in \cref{fig:walk-2-illustration}. 

        \begin{figure}[t!]\centering
            \begin{subfigure}[t!]{0.4\textwidth}
                \begin{center}
                    \begin{tikzpicture}[scale=1.5]
                        \node (0) at  (0, 0){\hspace{25pt}1};
                        \node (1) at  (0, 0)[circle,fill,color=black,inner sep=2pt]{};
                        \node (2) at  (-1, -1)[circle,fill,color=black,inner sep=2pt]{};
                        \node (3) at  (1, -1)[circle,fill,color=black,inner sep=2pt]{};
                        \node (4) at  (-1, -2)[circle,fill,color=black,inner sep=2pt]{};
                        \node (5) at  (-1, -3)[circle,fill,color=black,inner sep=2pt]{};
                        \node (6) at  (0.5, -2)[circle,fill,color=black,inner sep=2pt]{};
                        \node (7) at  (1.5, -2)[circle,fill,color=black,inner sep=2pt]{};
                        \node (8) at  (1.5, -3)[circle,fill,color=black,inner sep=2pt]{};
        
                        \draw[color=black, bend right = 10] (1) to (2) node[right]{\hspace{5pt}2};
                        \draw[color=black, bend right = 10] (2) to (1) node[right]{};
                        \draw[color=black, bend right = 25] (1) to (2) node[right]{};
                        \draw[color=black, bend right = 25] (2) to (1) node[right]{};

                        \draw[color=black, bend right = 10] (2) to (4) node[right]{\hspace{5pt}4};
                        \draw[color=black, bend right = 10] (4) to (2) node[right]{};

                        \draw[color=black, bend right = 10] (1) to (3) node[right]{\hspace{5pt}3};
                        \draw[color=black, bend right = 10] (3) to (1) node[right]{};
                        \draw[color=black, bend right = 25] (1) to (3) node[right]{};

                        \draw[color=black, densely dashed] (4) to (5) node[right]{\hspace{5pt}5};
                        \draw[color=black, densely dashed] (7) to (8) node[right]{\hspace{5pt}8};
                        \draw[color=black, densely dashed] (3) to (7) node[right]{\hspace{5pt}7};
                        \draw[color=black, densely dashed] (3) to (6) node[right]{\hspace{5pt}6};
                        \draw[color=black, densely dashed] (4) to (5) node[right]{\hspace{5pt}5};
    
                    \end{tikzpicture}
                \end{center}
            \end{subfigure}
            \begin{subfigure}[t!]{0.4\textwidth}
                \begin{center}
                    \begin{tikzpicture}[scale=1.5]
                        \node (0) at  (0, 0){\hspace{-55pt}1};
                        \node (1) at  (0, 0)[circle,fill,color=black,inner sep=2pt]{};
                        \node (2) at  (-1, -1)[circle,fill,color=black,inner sep=2pt]{};
                        \node (3) at  (1, -1)[circle,fill,color=black,inner sep=2pt]{};
                        \node (4) at  (-1, -2)[circle,fill,color=black,inner sep=2pt]{};
                        \node (5) at  (-1, -3)[circle,fill,color=black,inner sep=2pt]{};
                        \node (6) at  (0.5, -2)[circle,fill,color=black,inner sep=2pt]{};
                        \node (7) at  (1.5, -2)[circle,fill,color=black,inner sep=2pt]{};
                        \node (8) at  (1.5, -3)[circle,fill,color=black,inner sep=2pt]{};
        
                        \draw[color=black, bend right = 0] (7) to (8) node[right]{\hspace{5pt}8};
        
                        \draw[color=black, bend right = 0] (4) to (5) node[right]{\hspace{5pt}5};
        
                        \draw[color=black, bend right = 10] (2) to (4) node[right]{\hspace{5pt}4};
                        \draw[color=black, bend right = 10] (4) to (2) node[right]{};
                        \draw[color=black, bend right = 25] (2) to (4) node[right]{};
        
                        \draw[color=black, bend right = 10] (3) to (7) node[right]{\hspace{5pt}7};
                        \draw[color=black, bend right = 10] (7) to (3) node[right]{};
                        \draw[color=black, bend right = 25] (3) to (7) node[right]{};
        
                        \draw[color=black, bend right = 10] (1) to (2) node[right]{\hspace{-25pt}2};
                        \draw[color=black, bend right = 10] (2) to (1) node[right]{};
                        \draw[color=black, bend right = 25] (1) to (2) node[right]{};
                        \draw[color=black, bend right = 25] (2) to (1) node[right]{};
                        \draw[color=black, bend right = 35] (1) to (2) node[right]{};
                        
                        \draw[color=black, bend right = 0] (3) to (6) node[right]{\hspace{5pt}6};
        
                        \draw[color=black, bend right = 10] (1) to (3) node[right]{\hspace{20pt}3};
                        \draw[color=black, bend right = 10] (3) to (1) node[right]{};
                        \draw[color=black, bend right = 20] (1) to (3) node[right]{};
                        \draw[color=black, bend right = 20] (3) to (1) node[right]{};
                        \draw[color=black, bend right = 35] (1) to (3) node[right]{};
                        \draw[color=black, bend right = 35] (3) to (1) node[right]{};
                        \draw[color=black, bend right = 50] (1) to (3) node[right]{};
                        \draw[color=black, bend right = 50] (3) to (1) node[right]{};
                        \draw[color=black, bend right = 65] (1) to (3) node[right]{};
        
                    \end{tikzpicture}
                \end{center}
            \end{subfigure}
    
            \caption{(\textbf{left}) An illustration of the collection $\mathcal{Q}$ from the proof of \cref{lem:dist-identity} on a tree with eight nodes and root $1$. Only the walks $\{Q^{2, s}\}_{s=1, 2}, \{Q^{4, s}\}_{s=1, 2}, \{Q^{3, s}\}_{s=1, 2, 3}$ are shown. Repeated lines indicate the number of times an edge is traversed by the walks shown, i.e., with repetitions equal to $N_e(\mathcal{Q})$. Dashed lines indicate unused edges which are taken to be in the surrounding tree. (\textbf{right}) All walks in $\mathcal{Q} = \{Q^{j, s}\}_{j, s}$ are shown. }\label{fig:walk-2-illustration}
        \end{figure}
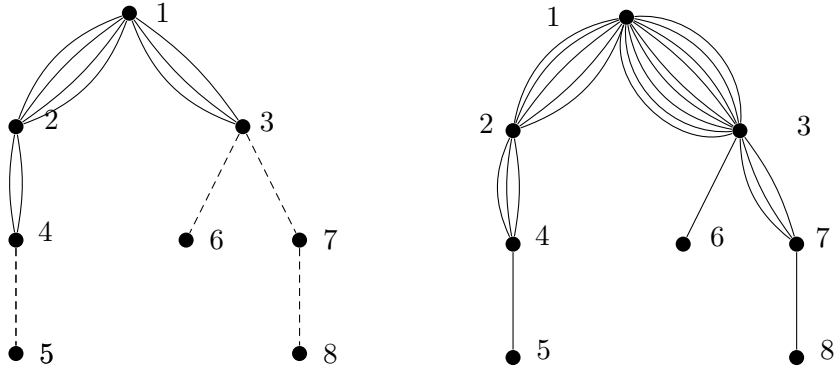

        We now claim that for each $e\in\emode_i$, it holds
            \begin{align}\label{eq:nep-neq}
                N_{e}(\mathcal{Q}) = N_{e}(\mathcal{P}) - 1.
            \end{align}
        If $e$ is a leaf edge, $N_{e}(\mathcal{P})=2$ and $N_{e}(\mathcal{Q}) = 1$ by construction. Assume $e = (j,k)\in\emode_i$ is not a leaf edge and that \cref{eq:nep-neq} holds for each $f\in \emode_i$ such that $e\prec f$. Then by induction and \cref{eq:recursion-q}, we have
            \begin{align*}
                N_{(j, k)}(\mathcal{Q}) &= \degc{k} +  \sum_{\substack{\ell\sim k \\ k\prec \ell}} N_{(k,\ell)}(\mathcal{Q})\\
                &= \degc{k} +  \sum_{\substack{\ell\sim k \\ k\prec \ell}} (N_{(k,\ell)}(\mathcal{P}) - 1)\\
                &= 1 +  \sum_{\substack{\ell\sim k \\ k\prec \ell}} N_{(k,\ell)}(\mathcal{P}) \\
                &= N_{(j, k)}(\mathcal{P}) - 1
            \end{align*}
        where the last equality follows from \cref{eq:recursion-p}. Therefore we have, by definition, \cref{eq:weight-p}, and \cref{eq:weight-q}, that
            \begin{align*}
                \sum_{j\neq 1} 2\dist{i, j} &= W(\mathcal{P})\\
                &= \sum_{e\in \emode_i} N_e(\mathcal{P}) \\
                &= \sum_{e\in \emode_i} (N_e(\mathcal{Q}) + 1)\\
                &= \frac{1}{2}\vol(G) + \sum_{j\neq 1} \dist{i, j} \degc{j}.
            \end{align*}
        The claim follows.
    \end{proof}

    The following result concerning $\kapdis_i$ can be obtained as a fairly straightforward consequence of \cref{lem:dist-identity}.

    \begin{corollary}\label{cor:kapd}
        Let $T=(V, E)$ be a combinatorial tree and $i\in V$. Then $\kapdis$ is given by
            \begin{align}
                \kapdis_i &= \frac{2n(2-\degc{i})}{\vol(G)} = \frac{n}{n-1}(2-\degc{i})
            \end{align}
    \end{corollary}
    
    \begin{proof}
        Let $i\in V$ be fixed. Observe by \cref{lem:dist-identity} that
            \begin{align*}
                \sum_{j\in V}\dist{i, j} (2-\degc{j}) &= \frac{1}{2}\vol(G).
            \end{align*}
        Therefore, $\kappa_j = \frac{2n(2-\degc{j})}{\vol(G)}$ satisfies
            \begin{align}\label{eq:kap-dist-tree}
                \mathcal{D}\kappa = n\mathbf{1}
            \end{align}
        as desired. On trees, the distance matrix $\mathcal{D}$ is nonsingular, so the solution $\kappa$ to \cref{eq:kap-dist-tree} is uniquely determined and the claim follows.
    \end{proof}
    
    \begin{remark}\label{rmk:steinerberger-weights}
        \cref{lem:dist-identity} can be extended to weighted trees if the shortest path metric $\Delta$ takes edge weights into account (via the path of least total weight). The proof amounts to a minor adjustment of the one presented herein (namely, by modifying~\cref{eq:step-count-defn}). In this setting, the shortest path distance matrix and Steinerberger curvature take on weighted variants. However, $\kapdis_i = \frac{2n(2-\degc{i})}{\vol(G)}$ would continue to hold. Thus in such a setting, $\kapdis_i$ would change only by a global scaling factor regardless of the weighting of the tree.
    \end{remark}

    \section{Proofs from the Introduction}\label{sec:comparisons-proof}

    \begin{proof}[Proof of~\cref{thm:comparison-1}]
        By~\cref{thm:lly-curv-combinatorial} and~\cref{thm:orc-curv-combinatorial}, the Lin--Lu--Yau curvature is given by the formula
            \begin{align*}
                \kaplly_{ij}&=2\left(\frac{1}{\degc{i}}+\frac{1}{\degc{j}}-1\right)\\
                &=\frac{1}{(1-\alpha)}\kaporc_{ij},
            \end{align*}
        from which the first part of~\cref{eq:comparison-1} follows. By~\cref{cor:kapd}, we have
            \begin{align*}
                \kapdis_{i} &= \frac{n}{n-1}\,(2-\degc{i}), \quad i\in V.
            \end{align*}
        Therefore
            \begin{align*}
                \frac{n-1}{n}\left(\frac{\kapdis_{i}}{\degc{i}}+\frac{\kapdis_{j}}{\degc{j}}\right)
                &=\frac{n-1}{n}\left(\frac{1}{\degc{i}}\cdot\frac{n}{n-1}(2-\degc{i})
                +\frac{1}{\degc{j}}\cdot\frac{n}{n-1}(2-\degc{j})\right)\\
                &=\frac{2-\degc{i}}{\degc{i}}+\frac{2-\degc{j}}{\degc{j}}
                =\left(\frac{2}{\degc{i}}-1\right)+\left(\frac{2}{\degc{j}}-1\right)\\
                &=2\left(\frac{1}{\degc{i}}+\frac{1}{\degc{j}}-1\right)
                =\kaplly_{ij}.
            \end{align*}
        The claim follows.
    \end{proof}

    \begin{proof}[Proof of~\cref{thm:comparison}]
        For \textit{(i)}, since $\{i, j\}$ is not a leaf edge, $\degc{i}, \degc{j} \geq 2$ and thus in particular $1/\degc{i} + 1/\degc{j} \leq 1 \leq \frac{1}{1-\alpha}$, so that the latter case of \cref{eq:orc-on-trees-comb} holds. Thus since $1-\alpha > 0$ as well,
            \begin{align*}
                \kaplly_{ij} &= 2\left(1/\degc{i}+1/\degc{j}-1\right)\\
                &\leq 2(1-\alpha)\left(1/\degc{i}+1/\degc{j}-1\right) = \kaporc_{ij}
            \end{align*}
        Next we have that for $x=i, j$, it holds
            \begin{align*}
                2\left(1/\degc{i}+1/\degc{j}-1\right) &> 2\left(1/\degc{x}-1\right)\\
                &= \frac{1}{\degc{x}} \left(2-2\degc{x} \right)\\
                &= \frac{2}{\degc{x}} \left(2-\degc{x} -1 \right)\\
                &\geq \frac{4}{\degc{x}} \left(\kapdis_{x} -\frac{1}{2}\right)
            \end{align*}
        For \textit{(ii)} we note first that by assumption,
            \begin{align*}
                \kaporc_{ij} &= 2(1-\alpha)\left(1/\degc{i}+1/\degc{j}-1\right) \\
                &= \frac{2(1-\alpha)}{\degc{j}}\\
                &\leq \frac{2}{\degc{j}} = \kaplly_{ij}
            \end{align*}
        Then we have that since $\degc{j}\geq 2$, it holds
            \begin{align*}
                \kaplly_{ij}&= \frac{2}{\degc{j}} \leq \frac{8}{3} \kapdis_i.
            \end{align*}
    \end{proof}

    \begin{proof}[Proof of \cref{thm:reverse-bm}]
        We have that
            \begin{align}\label{eq:reverse-bm-eq}
                \| n\mathbf{1}\|_\infty = \|\mathcal{D}\kapdis\|_\infty &\leq \|\mathcal{D}\|_{\ell_1\rightarrow \ell_\infty} \|\kapdis\|_1.
            \end{align}
        Here, $\|\mathcal{D}\|_{\ell_1\rightarrow \ell_\infty}$ denotes the norm of the matrix $\mathcal{D}$ as an operator from the normed space $(\mathbb{R}^n, \|\cdot\|_1)$ into the normed space  $(\mathbb{R}^n, \|\cdot\|_\infty)$. We recall the well-known fact that this norm is given by the entrywise maximum absolute value of the matrix, which in this case is $D$. Thus \cref{eq:reverse-bm-eq} reads
            \begin{align*}
                n \leq D\|\kapdis\|_1
            \end{align*}
        from which the claim follows.
    \end{proof}

    \section*{Statements \& Declarations}

    \subsection*{Acknowledgements}

    The author wishes to acknowledge Stefan Steinerberger for helpful conversations as well as the anonymous referee whose suggestions have greatly improved the paper.

    \subsection*{Financial Disclosures}

    The author wishes to acknowledge financial support from the Hal\i{}c\i{}o\u{g}lu Data Science Institute through its Graduate Prize Fellowship.  There are no other relevant financial or non-financial competing interests to disclose.
    
    \subsection*{Data Availability Statement}

    The author declares that the data supporting the findings of this study are available within the paper or at the publicly available repository \href{https://github.com/sawyer-jack-1/curvature-on-trees}{https://github.com/sawyer-jack-1/curvature-on-trees}.



\end{document}